\documentclass[onefignum,onetabnum]{siamart171218}

\usepackage{versions,afterpage}
\usepackage{subcaption,wrapfig}
\usepackage{lipsum}
\usepackage{amsfonts}
\usepackage{graphicx}
\usepackage{epstopdf}
\usepackage{algorithmic}
\usepackage{wrapfig}
\ifpdf
  \DeclareGraphicsExtensions{.eps,.pdf,.png,.jpg}
\else
  \DeclareGraphicsExtensions{.eps}
\fi

\newsiamremark{remark}{Remark}
\newsiamremark{hypothesis}{Hypothesis}
\crefname{hypothesis}{Hypothesis}{Hypotheses}
\newsiamthm{claim}{Claim}

\headers{Scaled Relative Graphs of Normal Matrices}{X. Huang, E. K. Ryu, and W. Yin}

\title{Scaled Relative Graphs of Normal Matrices 
}

 \author{Xinmeng Huang\thanks{Applied Mathematics and Computational Science, University of Pennsylvania (\email{xinmengh@sas.upenn.edu})}\and Ernest K. Ryu\thanks{Department of Mathematics, University of California, Los Angeles (\email{eryu@math.ucla.edu})} \and Wotao Yin\thanks{Department of Mathematics, University of California, Los Angeles (\email{wotaoyin@math.ucla.edu})}}
\usepackage{amsopn}

\ifpdf
\hypersetup{
  pdftitle={Scaled Relative Graphs of Normal Matrices},
  pdfauthor={X. Huang, E. K. Ryu, and W. Yin}
}
\fi

\usepackage{amsmath,amssymb,tikz,stmaryrd}
\usepackage{tikz}
\usetikzlibrary{fit, arrows, shapes, positioning, shadows, matrix, calc}
 
\usetikzlibrary{decorations.pathreplacing}
\tikzstyle std=[line width=0.7pt]   
\tikzstyle stdthin=[line width=0.3pt]   
\tikzstyle stdthick=[line width=1.0pt]   

\tikzstyle fwd=[line width=0.7pt, ->]   
\tikzstyle fwdthin=[line width=0.3pt, ->]   
\tikzstyle fwdthick=[line width=1.0pt, ->]   
\tikzstyle fwddash=[line width=0.7pt, dashed, ->]   

\tikzstyle bwd=[double, line width=0.3pt, ->]  
\tikzstyle refl=[double,  dashed, line width=0.2pt, ->]    

\tikzstyle{every node}=[font=\small] 

\tikzset{
  >=stealth', 
  invisible/.style={opacity=0}, 
  alt/.code args={<#1>#2#3}{\alt<#1>{\pgfkeysalso{#2}}{\pgfkeysalso{#3}}}, 
  visible on/.style={alt=#1{}{invisible}}, 
  smallnode/.style={circle, fill=black, thick, inner sep=1pt, minimum size=1.5pt}, 
  punkt/.style={
           rectangle,
           rounded corners,
           draw=black, very thick,
           text width=5.5em,
           minimum height=2em,
           text centered},
  punkt_big/.style={
           rectangle,
           rounded corners,
           draw=black, very thick,
           text width=7em,
           minimum height=2em,
           text centered},
}
\newtheorem{fact}{Fact}

  \usetikzlibrary{positioning}
  \usetikzlibrary{shadows}
  \usetikzlibrary{fit, arrows, shapes, positioning, shadows, matrix, calc}
\usetikzlibrary{decorations.pathreplacing}

\usepackage{graphicx}

\newcommand{\reals}{\mathbb{R}}
\newcommand{\complex}{\mathbb{C}}
\newcommand{\D}{\mathrm{Poly}}
\newcommand{\spanspan}{\mathrm{span}}

\newcommand{\C}{\mathrm{Circ}}
\newcommand{\disk}{\mathrm{Disk}}
\newcommand{\rac}{\mathrm{Arc}_\mathrm{min}}

\newcommand{\cG}{{\mathcal{G}}}
\newcommand{\Diag}{{\mathrm{Diag}}} 
\renewcommand{\Re}{\operatorname{Re}} 	
\renewcommand{\Im}{\operatorname{Im}}	

\usepackage{times,amsmath,graphicx,amssymb}
\usepackage{url}
\usepackage{tikz-3dplot}
\usetikzlibrary{shapes,calc,positioning}
\usetikzlibrary{calc,intersections}
\usepackage{tkz-euclide}

\makeatletter
\DeclareFontFamily{U}{tipa}{}
\DeclareFontShape{U}{tipa}{m}{n}{<->tipa10}{}
\newcommand{\arc@char}{{\usefont{U}{tipa}{m}{n}\symbol{62}}}%

\newcommand{\arc}[1]{\mathpalette\arc@arc{#1}}

\newcommand{\arc@arc}[2]{%
  \sbox0{$\m@th#1#2$}%
  \vbox{
    \hbox{\resizebox{\wd0}{\height}{\arc@char}}
    \nointerlineskip
    \box0
  }%
}
\makeatother

\definecolor{lightgrey}{gray}{0.8}
\definecolor{medgrey}{gray}{0.6}
\definecolor{darkgrey}{gray}{0.4}

\includeversion{arxiv}
\excludeversion{sirev}

\begin{document}

\maketitle

\begin{abstract}
The Scaled Relative Graph (SRG) is a geometric tool that maps the action of a multi-valued nonlinear operator onto the 2D plane, used to analyze the convergence of a wide range of iterative methods. As the SRG includes the spectrum for linear operators, we can view the SRG as a generalization of the spectrum to multi-valued nonlinear operators. In this work, we further study the SRG of linear operators and characterize the SRG of block-diagonal and normal matrices.
\end{abstract}



\section{Introduction}
The Scaled Relative Graph (SRG) is a geometric tool that maps the action of a multi-valued nonlinear operator onto the extended complex plane, analogous to how the spectrum maps the action of a linear operator to the complex plane. 
The SRG can be used to analyze convergence of a wide range of iterative methods expressed as fixed-point iterations.

\vspace{0.05in}
\noindent
\textbf{Scaled relative graph.}
For a matrix $A\in\reals^{n\times n}$, define $z_A\colon\reals^n\backslash\{0\}\rightarrow \complex$ with
\[
z_A(x)=\frac{\| Ax\|}{\| x\|}\exp[i\angle (Ax,x)],
\]
where \begin{align*}
    \angle(a,b) =
    \left\{
    \begin{array}{ll}
    \arccos\left(\tfrac{a^Tb}{\|a\|\|b\|}\right)&\text{ if }a\ne 0,\,b\ne 0\\
    0&\text{ otherwise}
    \end{array}
    \right.
\end{align*}
denotes the angle in $ [0,\pi]$ between $a$ and $b$.
The SRG of a matrix $A\in \reals^{n\times n}$ is
\[
\cG(A)=\left\{z_A(x),\overline{z_A}(x):x\in \reals^n,\,x\ne 0\right\}.
\]
This definition of the SRG, specific to (single-valued) linear operators, coincides with the more general definition for nonlinear multi-valued operators provided in \cite{ryu2019scaled}.
Ryu, Hannah, and Yin showed the SRG generalizes spectrum  in the following sense.
\begin{fact}[Theorem~3.1 of \cite{ryu2019scaled}]
\label{fact1}
If $A\in \reals^{n\times n}$ and $n = 1$ or $n\ge 3$, then $\Lambda(A)\subseteq \cG(A)$.
\end{fact}



2D geometric illustrations have been used 
by Eckstein and Bertsekas \cite{eckstein1989,eckstein1992}, 
Giselsson \cite{giselsson2017linear,giselsson_slides},
Banjac and Goulart \cite{banjac2018},
and Giselsson and Moursi \cite{giselsson_moursi2019}
to qualitatively understand the convergence of optimization algorithms.
The SRG was first presented as a rigorous formulation of such illustrations in Hannah and Yin's technical report \cite{hannahsrg} and was further expanded upon in the follow-up work by Ryu, Hannah, and Yin \cite{ryu2019scaled}.

\vspace{0.05in}
\noindent
\textbf{Contributions.}
Prior work \cite{ryu2019scaled,huang2019scaled} focused on the SRG of \emph{nonlinear} multi-valued operators.
For linear operators, Ryu, Hannah, and Yin \cite{ryu2019scaled} established $\cG(A)$ includes $\Lambda(A)$, as stated in Fact~\ref{fact1}, but did not characterize when and how $\cG(A)$ enlarges $\Lambda(A)$.
In this work, we further study the SRG of linear operators.
In particular, we fully characterize the SRG of block-diagonal and normal matrices as a certain polygon in hyperbolic (non-Euclidean) geometry, under the Poincar\'e half-plane model.

\vspace{0.05in}
\noindent
\textbf{Preliminaries.}
Let $A\in\reals^{n\times n}$.
Write $\Lambda(A)$ for the spectrum, the set of eigenvalues, of $A$.
$A$ is \emph{normal} if $A^TA=AA^T$.
Given matrices $A_1,\dots,A_m$, write $\Diag(A_1,\dots,A_m)$ for the block-diagonal matrix with $m$ blocks.
For $z\in \complex$, write $\overline{z}$ for its complex conjugate.
For a set $S\subseteq\mathbb{C}$, write $S^+=\{z\in S\,|\Im z\geq 0\}$.
In particular, write $\mathbb{C}^+=\{z\in \mathbb{C}\,|\Im z\geq 0\}$ and $\cG^+(A)=\left\{z_A(x):x\in \reals^n,\,x\ne 0\right\}$.
Note $z_A(x)\in \complex^+$ for all nonzero $x\in \reals^n$.
For $z_1,z_2\in \complex$, define 
\[
[z_1,z_2]=\{\theta z_1+(1-\theta)z_2:\theta\in[0,1]\},
\]
i.e., $[z_1,z_2]$ is the line segment connecting $z_1$ and $z_2$.

\section{Arc-edge polygon and arc-convexity}
Consider points $z_1,z_2\in \mathbb{C}^+$.
If $\Re z_1\neq \Re z_2$, let $ \C(z_1,z_2) $ be the circle in $\complex $ through $z_1$ and $z_2$ with the center on the real axis.
We can construct $ \C(z_1,z_2) $ by finding the center as the intersection of the perpendicular bisector of $[z_1,z_2]$ and the real axis. 
If $\Re z_1= \Re z_2$ but $z_1\ne z_2$, let $\C(z_1,z_2)$ be the line extending $[z_1,z_2]$. If $z_1=z_2$, then $\C(z_1,z_2)$ is undefined.
If $\Re z_1\neq \Re z_2$, let $\rac(z_1,z_2)\subseteq \mathbb{C}^+$ be the arc of $ \C(z_1,z_2)$ between $z_1$ and $z_2$ in the upper-half plane.
(If $\Im z_1>0$ or $\Im z_2>0$, then $\rac(z_1,z_2)\subseteq \mathbb{C}^+$ is the minor arc of $ \C(z_1,z_2)$ between $z_1$ and $z_2$.
If $\Im z_1=\Im z_2=0$, then $\rac(z_1,z_2)$ is a semicircle in $\mathbb{C}^+$.)
If $\Re z_1= \Re z_2$ but $z_1\ne z_2$, let $\rac(z_1,z_2)=[z_1,z_2]$.
If $z_1=z_2$, then $\rac(z_1,z_2)=\{z_1\}$.
For $z_1,z_2\in \mathbb{C}^+$ such that $\Re z_1\neq \Re z_2$, let $\disk(z_1,z_2)$ and $\disk^\circ (z_1,z_2)$ respectively be the closed and open disks enclosed by $ \C(z_1,z_2) $.
Figure~\ref{fig:circ_arc_def} illustrates these definitions.


\begin{figure}[h]
	\centering
	\begin{tikzpicture}[scale=0.75]
	\def\a{15};
	\def\b{75};
	\def\r{2.5}
	\def\t{2};
	
	\coordinate (A) at ({\t+\r*cos(\a)},{\r*sin(\a)});
	\coordinate (B) at ({\t+\r*cos(\b)},{\r*sin(\b)});
	
	\begin{scope}
	\draw [<->] (-1,0) -- (5,0);
	\draw [<->] (0,-2.5) -- (0,2.5);
	\end{scope}
	
	\draw [dashed] (2,0) circle (2.5);
	\draw (-1.2,2.3) node[below] {$\C(z_1,z_2)$};
	\def\tc{150};
	\draw [->] (-0.6,1.7) -- ({2+2.5*cos(\tc)},{2.5*sin(\tc)});
	
	\draw [line width = 1/0.75] (A) arc ({\a}:{\b}:{\r});
	\draw (A) node[right] {$z_2$};
	\draw (B) node[above] {$z_1$};
	
	\draw (4.7,2.3) node[above] {$\rac(z_1,z_2)$};
	\def\tr{60};
	\draw [->] (4.1,2.5) -- ({2+2.5*cos(\tr)},{2.5*sin(\tr)});
	
	\draw [] (A) -- (B);
	\draw [] (-0.5,-2.5) -- (4.2,2.2);
	
	\draw [] ({\t+(0.77*\r*cos(\a)+0.23*\r*cos(\b))+0.1},{(0.77*\r*sin(\a)+0.23*\r*sin(\b))+0.1}) -- ({\t+(0.77*\r*cos(\a)+0.23*\r*cos(\b))-0.1},{(0.77*\r*sin(\a)+0.23*\r*sin(\b))-0.1});
	\draw [] ({\t+(0.73*\r*cos(\a)+0.27*\r*cos(\b))+0.1},{(0.73*\r*sin(\a)+0.27*\r*sin(\b))+0.1}) -- ({\t+(0.73*\r*cos(\a)+0.27*\r*cos(\b))-0.1},{(0.73*\r*sin(\a)+0.27*\r*sin(\b))-0.1});
	
	\draw [] ({\t+(0.77*\r*cos(\b)+0.23*\r*cos(\a))+0.1},{(0.77*\r*sin(\b)+0.23*\r*sin(\a))+0.1}) -- ({\t+(0.77*\r*cos(\b)+0.23*\r*cos(\a))-0.1},{(0.77*\r*sin(\b)+0.23*\r*sin(\a))-0.1});
	\draw [] ({\t+(0.73*\r*cos(\b)+0.27*\r*cos(\a))+0.1},{(0.73*\r*sin(\b)+0.27*\r*sin(\a))+0.1}) -- ({\t+(0.73*\r*cos(\b)+0.27*\r*cos(\a))-0.1},{(0.73*\r*sin(\b)+0.27*\r*sin(\a))-0.1});
	
	\coordinate (C) at ({\t+0.5*(\r*cos(\a)+\r*cos(\b))},{0.5*(\r*sin(\a)+\r*sin(\b))});   
	\coordinate (O) at (2.0,0); 
	\tkzMarkRightAngle[size=0.2](O,C,A);
	
	\filldraw (\t,0) circle ({1.0/0.75pt});
	\filldraw (A) circle ({1.0/0.75pt});
	\filldraw (B) circle ({1.0/0.75pt});
	
	\end{tikzpicture}
	\begin{tikzpicture}[scale=0.75]
	\def\a{15};
	\def\b{75};
	\def\r{2.5}
	\def\t{1.5};
	
	\coordinate (A) at (\t,2);
	\coordinate (B) at (\t,0.5);
	
	\begin{scope}
	\draw [<->] (-1,0) -- (5,0);
	\draw [<->] (0,-2.5) -- (0,2.5);
	\end{scope}
	
	\draw [line width =1/0.7] (A) -- (B);
	\draw [dashed] (\t,-2.5) -- (\t,2.5);
	\draw (\t+0.3,-1) node[right] {$\C(z_1,z_2)$};
	\draw [->] (\t+0.4,-1) -- (\t,-1);
	
	\draw (A) node[right] {$z_1$};
	\draw (B) node[right] {$z_2$};
	
	\draw (\t+0.3,1) node[above right] {$\rac(z_1,z_2)$};
	\draw [->] (\t+0.4,1.35) -- (\t,1.35);
	
	\filldraw (\t,0) circle ({1.0/0.75pt});
	\filldraw (A) circle ({1.0/0.75pt});
	\filldraw (B) circle ({1.0/0.75pt});
	
	\end{tikzpicture}
	\caption{Illustration of $\C(z_1,z_2)$ and $\rac(z_1,z_2)$.
	}
	\label{fig:circ_arc_def}
\end{figure}
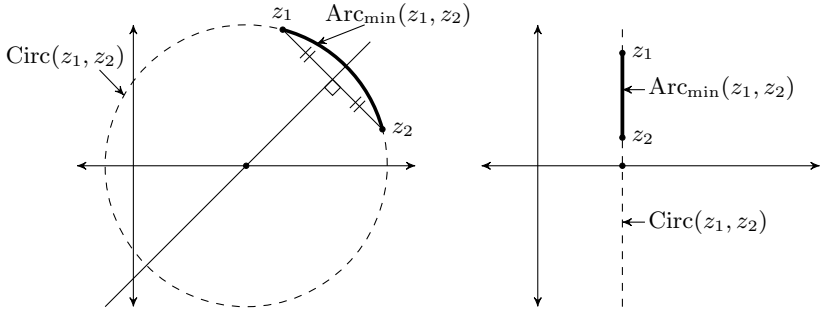

For $m\ge 1$ and $z_1,\dots,z_m\in\mathbb{C}^+$, we call $\D(z_1,z_2,\dots,z_m)$ an \emph{arc-edge polygon} and define it as follows.
For $m=1$, let $\D(z_1)=\{z_1\}$.
For $m\ge 2$, let
\begin{equation*}
S=\bigcup_{1\le i,j\le m}\rac(z_i,z_j)
\label{s:def}
\end{equation*}
and 
\begin{equation*}
\D(z_1,\dots,z_m)=S\cup \{\text{region enclosed by $S$}\}.
\label{poly:def}
\end{equation*}
Figure~\ref{fig:a-e-poly-example} illustrates this definition.
Note $\D(z_1,z_2)=\rac(z_1,z_2)$.
The ``region enclosed by $S$'' is the union of all regions enclosed by non-self-intersecting continuous loops (Jordan curves) within $S$.
Since $S$ is a connected set, we can alternatively define $\D(z_1,\dots,z_m)$ as the smallest simply connected set containing $S$.

\afterpage{%
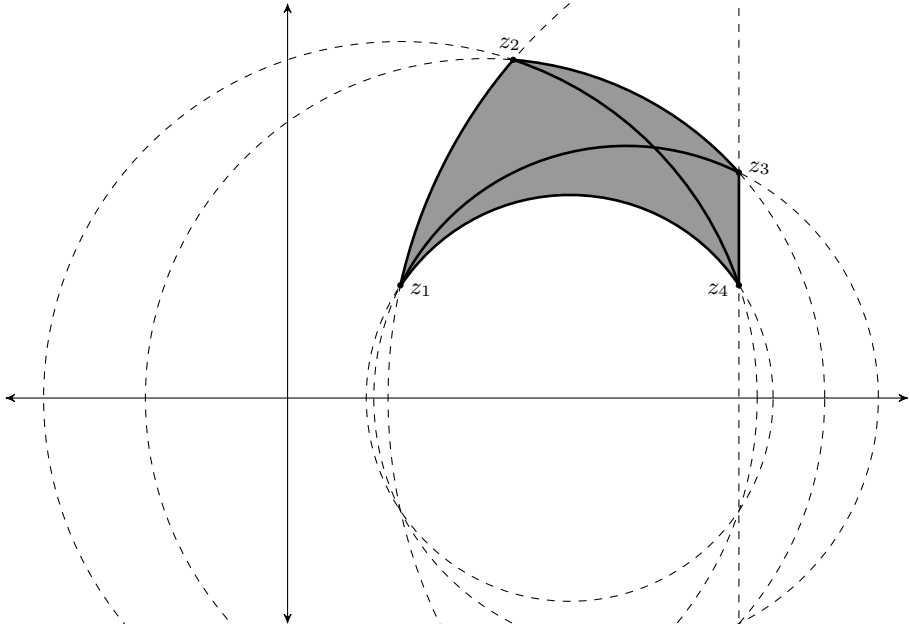
\begin{figure}[H]
	\begin{center}
		\begin{tabular}{c}
			\raisebox{-.5\height}{
				\begin{tikzpicture}[scale=1.5]
 				\clip (-2.55,-2) rectangle (5.55,3.5);
				
				\fill [fill=medgrey](4,1) -- (4,2) arc ({atan(2/(4-7/4))}:{atan(3/(2-7/4))}:{sqrt((4-7/4)^2+2^2)})--(2,3) arc (139.3987:167.4712:4.61);
				\fill[fill=white](2.5,0) circle (1.803);

				\draw [line width =1] (4,1) arc (18.4349:71.5651:3.1623);
				\draw [line width =1] (4,2) arc ({atan(2/(4-7/4))}:{atan(3/(2-7/4))}:{sqrt((4-7/4)^2+2^2)});
				\draw [line width =1] (4,2)--(4,1);
				\draw [line width =1] (4,2) arc ({atan(2/(4-3))}:{180-atan(1/(3-1))}:{sqrt(1+(3-1)^2)});
				\draw [line width =1] (2,3) arc (139.3987:167.4712:4.61);
				\draw [line width=1] (4,1) arc (33.7:146.3:1.8);
				
				\draw[dashed] (4,-3.5) -- (4,3.5);
				\draw[dashed] (3,0) circle ({sqrt(2^2+1)});
				\draw[dashed] (1,0) circle ({sqrt(3^2+1^2)});
				\draw[dashed] (2.5,{3.5}) arc ({180-atan(3.5/(5.5-2.5))}:{180+atan(3.5/(5.5-2.5))}:{sqrt(4.5^2+1)});
				\draw[dashed] (2.5,0) circle ({sqrt(1+1.5^2)});
				\draw[dashed] (7/4,0) circle ({sqrt(2^2+(4-7/4)^2)});
				
				\draw (1,1.1) node[below right] {$z_1$};
				\draw (2.15,3) node[above left] {$z_2$};
				\draw (4,1.9) node[above right] {$z_3$};
				\draw (4,1.1) node[below left] {$z_4$};
				
				\filldraw (1,1) circle ({1.0/1.5pt});
				\filldraw (2,3) circle ({1.0/1.5pt});
				\filldraw (4,2) circle ({1.0/1.5pt});
				\filldraw (4,1) circle ({1.0/1.5pt});
				
				\draw [<->] (-2.5,0) -- (5.5,0);
				\draw [<->] (0,-2) -- (0,3.5);
				\end{tikzpicture}}
		\end{tabular}
	\end{center}
	\caption{ The shaded region illustrates the arc-edge polygon $\D(z_1,z_2,z_3,z_4)$ for $z_1=1+i$, $z_2=2+3i$, $z_3=4+2i$, and $z_4=4+i$. The solid arcs illustrate $\rac(z_i,z_j)$ and the dashed circles illustrate $\C(z_i,z_j)$ for $i,j=1,\dots,m$.}
	\label{fig:a-e-poly-example}
\end{figure}
 \begin{figure}[h]
 	\begin{center}
		\raisebox{-.5\height}{
 		\begin{tikzpicture}[scale=1.9]

 		\fill [fill=medgrey] (2,1) -- (2,2) arc ({acos(sqrt(5)/5)}:{acos((-2)*sqrt(5)/5)}:{sqrt(5)}) -- (-1,1) arc ({90:0:1}) -- (0,0) arc (180:{atan(0.5/0.375)}:0.625) -- (1,0.5) arc ({180-atan(0.5/(15/8-1))}:{atan(1/(2-15/8))}:{sqrt((7/8)^2+0.5^2)});

 		\draw[] (2,2) arc ({acos(sqrt(5)/5)}:{acos((-2)*sqrt(5)/5)}:{sqrt(5)});
 		\draw[] (1,0.5) arc ({atan(0.5/0.375)}:180:0.625);
 		\draw[] (0,0) arc (0:90:1); 
 		\draw[] (2,1) -- (2,2);
 		\draw[] (2,1) arc ({atan(1/(2-15/8))}:{180-atan(0.5/(15/8-1))}:{sqrt((7/8)^2+0.5^2)});

 		\draw[] (2,2) arc ({180-atan(2/(27/8-2))}:{180-atan(0.5/(27/8-1))}:{sqrt((19/8)^2+(1/2)^2)}); 
 		\draw[] (2,1) arc ({atan(1/1.5)}:{180-atan(1/1.5)}:{sqrt(1.5^2+1)});
 		\draw[] (2,1) arc ({atan(1/0.75)}:180:1.25);
 		\draw[] (2,2) arc (90:180:2);
 		\draw[] (1,0.5) arc ({atan(0.5/(1+3/16))}:{180-atan(1/(-3/16+1))}:{sqrt((1+3/16)^2+0.5^2)});

 		\draw[] (0.5,{sqrt(19/4)}) arc ({180-atan(sqrt(19/4)/4.5)}:180:5);
 		\draw[] (1,0.5) arc ({atan(0.5/(1+15/4))}:{atan(sqrt(19/4)/(0.5+15/4))}:{sqrt((17/4)^2+19/4)});
 		\draw[] (2,1) arc ({atan(1/2)}:{atan(sqrt(19/4)/0.5)}:{sqrt(5)});

 		\draw[] (0,0) arc (0:{atan(1/0.75)}:{1.25});
 		\draw[] (-0.5,1) arc ({atan(1/0.25)}:{180-atan(1/0.25)}:{sqrt(1+0.25^2)});
 		\draw[] (1,0.5) arc ({atan(0.5)}:{180-atan(2)}:{sqrt(5/4)});
 		\draw[] (2,1) arc ({atan(1/1.25)}:{180-atan(1/1.25)}:{sqrt(1+1.25^2)});
 		\draw[] (2,2) arc ({atan(2/(2-27/20))}:{180-atan(1/(27/20+0.5))}:{sqrt(4+(13/20)^2)});

 		\filldraw (-1,1) circle (1.0/1.9pt);
 		\filldraw (0,0) circle (1.0/1.9pt);
 		\filldraw (1,0.5) circle (1.0/1.9pt);
 		\filldraw (2,2) circle (1.0/1.9pt);
		\filldraw (2,1) circle (1.0/1.9pt);
 		\filldraw (0.5,{sqrt(19/4)}) circle (1.0/1.9pt);

 		\filldraw (-0.5,1) circle (1.0/1.9pt);

 		\draw (-0.12,0.1) node {$z_1$};
 		\draw (1.07,0.4) node {$z_2$};
 		\draw (2.1,0.9) node {$z_3$};
 		\draw (2.1,2.1) node {$z_4$};
 		\draw (0.45,2.27) node {$z_5$};
 		\draw (-1.05,0.9) node {$z_6$};
 		\draw (-0.29,0.97) node {$z_7$};
		
		\draw [dashed] (-1.2,0) -- (2.2,0);
		
 		\end{tikzpicture}}
\!$\stackrel{f\circ g}{\longrightarrow}$\!
		\raisebox{-.5\height}{
 		\begin{tikzpicture}[scale=2.4]
		\draw [dashed] (0,0) circle (1);
	\filldraw[fill=white] (1,0) circle (1.0/2.4pt);

	\coordinate (z1) at (-1,0);
	\coordinate (z2) at (1/9,-8/9);
	\coordinate (z3) at (2/3,-2/3);
	\coordinate (z4) at (7/9,-4/9);
	\coordinate (z5) at (2/3,-1/6);
	\coordinate (z6) at (1/3,2/3);
	\coordinate (z7) at (1/9,4/9);
	
	\fill[medgrey] (z1) -- (z2) -- (z3) -- (z4) -- (z5) -- (z6);

	\draw [] (z1) -- (z2);
	\draw [] (z2) -- (z3);
	\draw [] (z3) -- (z4);
	\draw [] (z4) -- (z5);
	\draw [] (z5) -- (z6);
	\draw [] (z6) -- (z1);
	
	\draw [] (z1) -- (z3);
	\draw [] (z1) -- (z4);
	\draw [] (z1) -- (z5);
	\draw [] (z1) -- (z6);
	\draw [] (z1) -- (z7);
	\draw [] (z2) -- (z4);
	\draw [] (z2) -- (z5);
	\draw [] (z2) -- (z6);
	\draw [] (z2) -- (z7);
	\draw [] (z3) -- (z5);
	\draw [] (z3) -- (z6);
	\draw [] (z3) -- (z7);
	\draw [] (z4) -- (z7);
	\draw [] (z5) -- (z6);
	\draw [] (z6) -- (z7);

		\filldraw (z1) circle (1.0/2.4pt);
		\filldraw (z2) circle (1.0/2.4pt);
		\filldraw (z3) circle (1.0/2.4pt);
		\filldraw (z4) circle (1.0/2.4pt);
		\filldraw (z5) circle (1.0/2.4pt);
		\filldraw (z6) circle (1.0/2.4pt);
		\filldraw (z7) circle (1.0/2.4pt);
		
		\draw (-1.1,-0.05) node {$w_1$};
		\draw (-0.0,-0.92) node {$w_2$};
		\draw (0.8,-0.8) node {$w_3$};
		\draw (1,-0.5) node {$w_4$};
		\draw (0.8,-0.15) node {$w_5$};
		\draw (0.33,0.74) node {$w_6$};
		\draw (0.02,0.29) node {$w_7$};

 		\end{tikzpicture}}
 	\end{center}
 	\caption{Illustration of $f\circ g$ and  Lemma~\ref{lem:arc-convex2}. The one-to-one map $f\circ g$ of \eqref{eq:fg-klein} maps $\D(z_1,\dots,z_7)$ (a hyperbolic polygon) into a Euclidean polygon. We denote the mapped points as $w_i=f(g(z_i))$ for $i=1,\dots,7$. The equivalent Euclidean geometry tells us that $\D(z_1,\dots,z_7)$ is ``convex''  and can be enclosed by the curve through $z_1\rightarrow z_2\rightarrow z_3\rightarrow z_4\rightarrow z_6\rightarrow z_1$.
Note that $z_5$ and $z_7$ are not necessary in the description of the boundary.
}
 	\label{fig:hyperbolic}
 \end{figure}
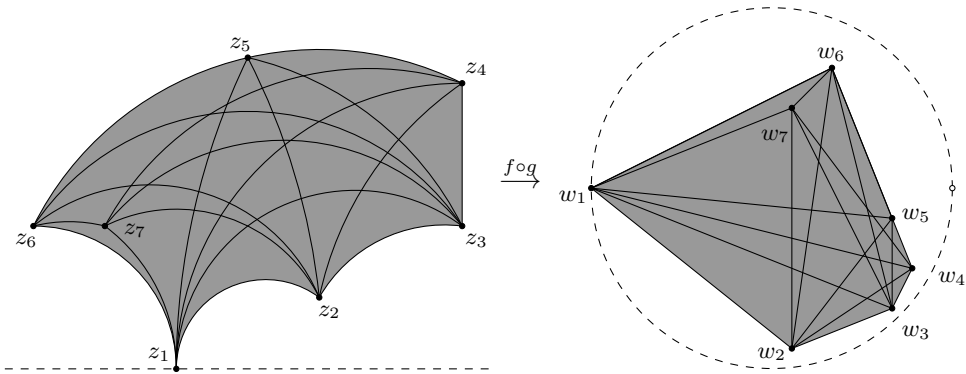
 \clearpage
 }


This construction of $\rac$ gives rise to the classical \emph{Poincar\'e half-plane model} of hyperbolic (non-Euclidean) geometry, where a $\rac(z_1,z_2)$ and $\C(z_1,z_2)\cap \complex^+$ are, respectively, the ``line segment'' between $z_1$ and $z_2$ and the ``line'' through $z_1$ and $z_2$ in the hyperbolic space \cite{beltrami_eugenio_1868_2243105,Poincare}.
The \emph{Beltrami--Klein model} maps the Poincar\'e half-plane model onto the unit disk and $\rac$ to straight line segments \cite{Klein1873,Beltrami1868}.
Specifically, the one-to-one map 
\begin{equation}
f\circ g\colon \complex^+\rightarrow \{z\in \complex:|z|\le 1,\,z\ne 1\}
\label{eq:fg-klein}
\end{equation}
defined by 
\[
f(z)= \frac{2z}{1+|z|^2},
\qquad
g(z)= \frac{z-i}{z+i}
\]
maps the Poincar\'e half-plane model to the Beltrami--Klein model
while mapping hyperbolic line segments $\rac$ into Euclidean straight line segments.
The Beltrami--Klein model demonstrates that any qualitative statement about convexity in the Euclidean plane is equivalent to an analogous statement in the Poincar\'e half-plane model.
See Figure~\ref{fig:hyperbolic}.

\begin{lemma}
\label{lem:arc-convex2}
Let $z_1,\dots,z_m\in\complex^+$ and $m\ge 1$.
Then $\D(z_1,\dots,z_m)$ is ``convex'' in the following non-Euclidean sense:
\[
w_1,w_2\in \D(z_1,\dots,z_m)
\quad\Rightarrow\quad
\rac(w_1,w_2)\subseteq \D(z_1,\dots,z_m).
\]
If $\D(z_1,\dots,z_m)$ has an interior, then there is $\{\zeta_1,\dots,\zeta_q\}\subseteq\{z_1,\dots,z_m\}$ such that 
\[
\rac(\zeta_1,\zeta_2)\cup \rac(\zeta_2,\zeta_3)\cup \cdots \cup 
\rac(\zeta_{q-1},\zeta_q)\cup \rac(\zeta_q,\zeta_1)
\]
is a Jordan curve, and the region the curve encloses is $\D(z_1,\dots,z_m)$.
\end{lemma}
\begin{proof}
Let $w_1,\dots,w_m$ be in the unit complex disk. Consider the construction 
\[
\tilde{S}=\bigcup_{1\le i,j\le m}[w_i,w_j]
\]
and 
\[
\widetilde{\mathrm{Poly}}(w_1,\dots,w_m)=
\tilde{S}\cup \{\text{region enclosed by $\tilde{S}$}\}.
\]
This is the (Euclidean) 2D polyhedron given as the convex hull of $w_1,\dots,w_m$.
The Euclidean convex hull has the properties analogous to those in the Lemma statement, and we use the map $(f\circ g)^{-1}$, where $f\circ g$ is as given by \eqref{eq:fg-klein} to map the properties to our setup.
\end{proof}

\section{SRGs of block-diagonal matrices}
We characterize the SRG of block-diagonal matrices as follows.
\begin{theorem}
\label{mat-main-thm}
Let $A_1,\dots,A_m$ be square matrices, where $m\ge 1$.
Then
\[
\cG^+\left(\Diag(A_1,\dots,A_m)\right)
=\bigcup_{\substack{z_i\in \mathcal{G}^+(A_i)\\
i=1,\dots,m}}\D(z_1,\dots,z_m).
\]
\end{theorem}
\begin{proof}
When $m=1$, there is nothing to show.
Assume $m\ge 2$.


\vspace{0.05in}
\noindent\textbf{Step 1.}
Let $A_i\in\reals^{n_i\times n_i}$  and $u_i\in \reals^{n_i}$ for $i=1,\dots,m$.
We use the notation $n=n_1+\dots+n_m$,
\[
\mathbf{u}=
\begin{bmatrix}
u_1\\\vdots\\u_m
\end{bmatrix}\in \reals^n,
\qquad
\mathbf{u}_i=
\begin{bmatrix}
0\\\vdots\\0\\u_i\\0\\\vdots\\0
\end{bmatrix}\in \reals^n\quad\text{for }i=1,\dots,m,
\]
and $\mathbf{A}=\Diag(A_1,\dots,A_m)\in\reals^{n\times n}$.
Then we have 
\begin{align}
\cG^+\left(\Diag(A_1,\dots,A_m)\right)&=
\bigcup_{\substack{\mathbf{u}\in \reals^{n}\backslash\{0\}}}z_\mathbf{A}(\mathbf{u})\nonumber\\
&=
\bigcup_{\substack{u_i\in\reals^{n_i},\,u_i\ne 0\\i=1,\dots,m}}z_\mathbf{A}\left(\spanspan(\mathbf{u}_1,\dots,\mathbf{u}_m)\backslash\{0\}\right)
\label{thm:first_eq}
\end{align}
and 
\begin{align}
\bigcup_{\substack{u_i\in\reals^{n_i},\,u_i\ne 0\\i=1,\dots,m}}\D\left(z_{\mathbf{A}}(\mathbf{u}_1),\dots,z_{\mathbf{A}}(\mathbf{u}_m)\right)
&=
\bigcup_{\substack{u_i\in\reals^{n_i},\,u_i\ne 0\\i=1,\dots,m}}\D\left(z_{A_1}(u_1),\dots,z_{A_m}(u_m)\right)
\nonumber\\
&=\bigcup_{\substack{z_i\in \cG^+(A_i)\\i=1,\dots,m}}\D\left(z_1,\dots,z_m\right).
\label{thm:second_eq}
\end{align}
To clarify, $\mathbf{u}_i$ depends on $u_i$ for $i=1,\dots,m$.
In the following, we show
\begin{equation}
z_\mathbf{A}\left(\spanspan(\mathbf{u}_1,\dots,\mathbf{u}_m)\backslash\{0\}\right)
=
\D\left(z_{\mathbf{A}}(\mathbf{u}_1),\dots,z_{\mathbf{A}}(\mathbf{u}_m)\right)
\label{thm:third_eq}
\end{equation}
for all $\mathbf{u}_i$ given by $u_i\ne 0$ for $i=1,\dots,m$.
Once \eqref{thm:third_eq} is proved, \eqref{thm:first_eq} and \eqref{thm:second_eq} are equivalent and the proof is complete.


\vspace{0.05in}
\noindent\textbf{Step 2.}
We show the following intermediate result:
for all nonzero $\mathbf{u},\mathbf{v}\in \reals^n$ such that
\begin{equation}
\langle \mathbf{u},\mathbf{v}\rangle=\langle \mathbf{A}\mathbf{u},\mathbf{v}\rangle=\langle \mathbf{u},\mathbf{A}\mathbf{v}\rangle=\langle \mathbf{A}\mathbf{u},\mathbf{A}\mathbf{v}\rangle=0,
\label{eq:step2-cond} 
\end{equation}
we have
\begin{equation}
z_\mathbf{A}(\spanspan(\mathbf{u},\mathbf{v})\backslash \{0\})= \rac(z_\mathbf{A}(\mathbf{u}),z_\mathbf{A}(\mathbf{v})).
\label{eq:step2}
\end{equation}

First, consider the case $\Re z_\mathbf{A}(\mathbf{u})\ne \Re z_\mathbf{A}(\mathbf{v})$.
Let the circle $\C(z_\mathbf{A}(\mathbf{u}),z_\mathbf{A}(\mathbf{v}))$ be centered at $(t,0)$ with $t\in \reals$ and radius $r\ge 0$.
Then $z_\mathbf{A}(\mathbf{u})$ and $z_\mathbf{A}(\mathbf{v})$ satisfy
\begin{align*}
(\Re z_\mathbf{A}(\mathbf{u})-t)^2+(\Im z_\mathbf{A}(\mathbf{u}))^2&=r^2\\
(\Re z_\mathbf{A}(\mathbf{v})-t)^2+(\Im z_\mathbf{A}(\mathbf{v}))^2&=r^2.
\end{align*}
This is equivalent to
\begin{align*}
&\langle \mathbf{A} \mathbf{u},\mathbf{A} \mathbf{u}\rangle -2t\langle \mathbf{A} \mathbf{u}, \mathbf{u}\rangle +(t^2-r^2)\langle  \mathbf{u}, \mathbf{u}\rangle=0\\
&\langle \mathbf{A} \mathbf{v},\mathbf{A} \mathbf{v}\rangle -2t\langle \mathbf{A} \mathbf{v}, \mathbf{v}\rangle +(t^2-r^2)\langle  \mathbf{v}, \mathbf{v}\rangle=0.
\end{align*}
Let $\alpha_1,\alpha_2\in\reals$ and $\mathbf{w}=\alpha_1\mathbf{u}+\alpha_2\mathbf{v}$. Assume $\mathbf{w}\ne 0$.
Using \eqref{eq:step2-cond} and basic calculations, we have
\[
\langle \mathbf{A}\mathbf{w} ,\mathbf{A} \mathbf{w}\rangle -2t\langle \mathbf{A} \mathbf{w}, \mathbf{w}\rangle +(t^2-r^2)\langle  \mathbf{w}, \mathbf{w}\rangle=0,
\]
and this is equivalent to
\begin{align*}
(\Re z_\mathbf{A}(\mathbf{w})-t)^2+(\Im z_\mathbf{A}(\mathbf{w}))^2&=r^2.
\end{align*}
Therefore 
\[
z_\mathbf{A}(\mathbf{w})=z_\mathbf{A}(\alpha_1\mathbf{u}+\alpha_2\mathbf{v})\in \C(z_\mathbf{A}(\mathbf{u}),z_\mathbf{A}(\mathbf{v})).
\]
Notice that
\[
\Re\,z_\mathbf{A}(\mathbf{w})=\frac{\langle \mathbf{A}\mathbf{w} , \mathbf{w}\rangle}{\langle\mathbf{w}, \mathbf{w}\rangle}=\frac{\alpha_1^2\langle \mathbf{A}\mathbf{u},\mathbf{u}\rangle+\alpha_2^2\langle \mathbf{A}\mathbf{v},\mathbf{v}\rangle}{\alpha_1^2\langle\mathbf{u} , \mathbf{u}\rangle+\alpha_2^2\langle\mathbf{v} , \mathbf{v}\rangle}
\]
fills the interval $[\Re\,z_A(\mathbf{u}),\Re\,z_A(\mathbf{v})]$ as $\alpha_1$ and $\alpha_2$ varies. So we have
\[
\bigcup_{
\substack{\alpha_1,\alpha_2\in\reals\\
\alpha_1\mathbf{u}+\alpha_2\mathbf{v}\ne 0}}
z_\mathbf{A}(\alpha_1\mathbf{u}+\alpha_2\mathbf{v})= \rac(z_\mathbf{A}(\mathbf{u}),z_\mathbf{A}(\mathbf{v}))
\]
and we conclude \eqref{eq:step2}.

Next, consider the case $\Re z_\mathbf{A}(\mathbf{u})= \Re z_\mathbf{A}(\mathbf{v})$.
Note that
\[
\Re z_\mathbf{A}(\mathbf{u})=
\frac{\langle \mathbf{A}\mathbf{u} , \mathbf{u}\rangle}{\langle\mathbf{u}, \mathbf{u}\rangle},
\qquad
\Re z_\mathbf{A}(\mathbf{v})=\frac{\langle \mathbf{A}\mathbf{v} , \mathbf{v}\rangle}{\langle\mathbf{v}, \mathbf{v}\rangle}.
\]
Let $\alpha_1,\alpha_2\in\reals$ and $\mathbf{w}=\alpha_1\mathbf{u}+\alpha_2\mathbf{v}$.
Assume $\mathbf{w}\neq0$.
Using \eqref{eq:step2-cond} and basic calculations, we have 
\[
\Re\,z_\mathbf{A}(\mathbf{w})=\frac{\langle \mathbf{A}\mathbf{w} , \mathbf{w}\rangle}{\langle\mathbf{w}, \mathbf{w}\rangle}=\Re\,z_\mathbf{A}(\mathbf{u})=\Re\,z_\mathbf{A}(\mathbf{v}).
\]
Notice that
\[
|z_\mathbf{A}(\mathbf{w})|^2=\frac{\langle \mathbf{A}\mathbf{w},\mathbf{A}\mathbf{w}\rangle}{\langle \mathbf{w},\mathbf{w}\rangle}=\frac{\alpha_1^2\langle \mathbf{A}\mathbf{u},\mathbf{A}\mathbf{u}\rangle+\alpha_2^2\langle \mathbf{A}\mathbf{v},\mathbf{A}\mathbf{v}\rangle}{\alpha_1^2\langle\mathbf{u} , \mathbf{u}\rangle+\alpha_2^2\langle\mathbf{v} , \mathbf{v}\rangle}
\]
fills the interval $[|z_\mathbf{A}(\mathbf{u})|^2,|z_\mathbf{A}(\mathbf{v})|^2]$ as $\alpha_1$ and $\alpha_2$ varies.
So $\Im\,z_\mathbf{A}(\mathbf{w})$ fills the interval $[\Im\,z_\mathbf{A}(\mathbf{u}),\Im\,z_\mathbf{A}(\mathbf{v})]$ as $\alpha_1$ and $\alpha_2$ varies, and we conclude
\[
z_\mathbf{A}(\spanspan(\mathbf{u},\mathbf{v})\backslash \{0\})=[z_\mathbf{A}(\mathbf{u}),z_\mathbf{A}(\mathbf{v})]= \rac(z_\mathbf{A}(\mathbf{u}),z_\mathbf{A}(\mathbf{v})).
\]




\vspace{0.05in}
\noindent\textbf{Step 3.}
We show
\begin{equation}
z_\mathbf{A}\left(\spanspan(\mathbf{u}_1,\dots,\mathbf{u}_m)\backslash\{0\}\right)
\subseteq
\D\left(z_{\mathbf{A}}(\mathbf{u}_1),\dots,z_{\mathbf{A}}(\mathbf{u}_m)\right)
\label{eq:step6-inclusion}
\end{equation}
by induction. Clearly
\[
z_\mathbf{A}\left(\spanspan(\mathbf{u}_1)\backslash\{0\}\right)
=
\D\left(z_{\mathbf{A}}(\mathbf{u}_1)\right).
\]
Now assume \eqref{eq:step6-inclusion} holds for $m-1$.
By \eqref{eq:step2}, we have
\[
z_\mathbf{A}\left(\spanspan(\mathbf{u}_1,\dots,\mathbf{u}_m)\backslash\{0\}\right)=
 \bigcup_{\zeta\in z_\mathbf{A}\left(\spanspan(\mathbf{u}_1,\dots,\mathbf{u}_{m-1})\backslash\{0\}\right)}\rac(\zeta,z_{\mathbf{A}}(\mathbf{u}_m)).
\]
By the induction hypothesis, $\zeta\in z_\mathbf{A}\left(\spanspan(\mathbf{u}_1,\dots,\mathbf{u}_{m-1})\backslash\{0\}\right)$, implies
\[
\zeta\in 
\D\left(z_{\mathbf{A}}(\mathbf{u}_1),\dots,z_{\mathbf{A}}(\mathbf{u}_{m-1})\right)
\subseteq
\D\left(z_{\mathbf{A}}(\mathbf{u}_1),\dots,z_{\mathbf{A}}(\mathbf{u}_m)\right).
\]
By construction,
\[
z_{\mathbf{A}}(\mathbf{u}_m)\in 
\D\left(z_{\mathbf{A}}(\mathbf{u}_1),\dots,z_{\mathbf{A}}(\mathbf{u}_m)\right).
\]
``Convexity'' of Lemma~\ref{lem:arc-convex2} implies
\[
 \bigcup_{\zeta\in z_\mathbf{A}\left(\spanspan(\mathbf{u}_1,\dots,\mathbf{u}_{m-1})\backslash\{0\}\right)}\rac(\zeta,z_{\mathbf{A}}(\mathbf{u}_m))
 \subseteq \D\left(z_{\mathbf{A}}(\mathbf{u}_1),\dots,z_{\mathbf{A}}(\mathbf{u}_m)\right),
\]
and we conclude \eqref{eq:step6-inclusion}.

\vspace{0.05in}
\noindent\textbf{Step 4.}
We show
\begin{equation}
z_\mathbf{A}\left(\spanspan(\mathbf{u}_1,\dots,\mathbf{u}_m)\backslash\{0\}\right)
\supseteq
\D\left(z_{\mathbf{A}}(\mathbf{u}_1),\dots,z_{\mathbf{A}}(\mathbf{u}_m)\right).
\label{eq:step3-identity}
\end{equation}

First, consider the case where $\D\left(z_{\mathbf{A}}(\mathbf{u}_1),\dots,z_{\mathbf{A}}(\mathbf{u}_m)\right)$ has no interior.
In 2D Euclidean geometry, a polygon has no interior when it is a single line segment or a point.
The Beltrami--Klein model provides us with an equivalent statement in hyperbolic geometry: the ``polygon''  can be expressed as $\D\left(z_{\mathbf{A}}(\mathbf{u}_1),\dots,z_{\mathbf{A}}(\mathbf{u}_m)\right)=\rac(z_{\mathbf{A}}(\boldsymbol{\mu}_1),z_{\mathbf{A}}(\boldsymbol{\mu}_2))$ where $\boldsymbol{\mu}_1,\boldsymbol{\mu}_2\in \{\mathbf{u}_1,\dots,\mathbf{u}_m\}$.
By the reasoning of Step 2, we conclude 
\begin{align*}
z_\mathbf{A}\left(\spanspan(\mathbf{u}_1,\dots,\mathbf{u}_m)\backslash\{0\}\right)
&\supseteq
z_\mathbf{A}\left(\spanspan(\boldsymbol{\mu}_1,\boldsymbol{\mu}_2)\backslash\{0\}\right)\\
&=
\rac(z_{\mathbf{A}}(\boldsymbol{\mu}_1),z_{\mathbf{A}}(\boldsymbol{\mu}_2))=\D\left(z_{\mathbf{A}}(\mathbf{u}_1),\dots,z_{\mathbf{A}}(\mathbf{u}_m)\right).
\end{align*}

Next, consider the case where $\D\left(z_{\mathbf{A}}(\mathbf{u}_1),\dots,z_{\mathbf{A}}(\mathbf{u}_m)\right)$ has an interior. 
In this case, $\dim \spanspan(\mathbf{u}_1,\dots,\mathbf{u}_m)\ge 3$ by the arguments of Step 2.
Assume for contradiction that there is a  $z\in  \D\left(z_{\mathbf{A}}(\mathbf{u}_1),\dots,z_{\mathbf{A}}(\mathbf{u}_m)\right)$ but $z\notin z_\mathbf{A}\left(\spanspan(\mathbf{u}_1,\dots,\mathbf{u}_m)\backslash\{0\}\right)$.
Let $\zeta_1,\dots,\zeta_q$ be vertices provided by Lemma~\ref{lem:arc-convex2}.
There exists corresponding  $\{\boldsymbol{\mu}_1,\dots,\boldsymbol{\mu}_q\}\subseteq\{\textbf{u}_1,\dots,\textbf{u}_m\}$ such that $\zeta_i=z_\mathbf{A}(\boldsymbol{\mu}_i)$
for $i=1,\dots,q$.
Define the curve 
\[
\boldsymbol{\eta}(t):[1,q+1]\rightarrow \spanspan(\mathbf{u}_1,\dots,\mathbf{u}_m)\cap S^{m-1},
\]
where $S^{m-1}\subset \reals^m$ is the unit sphere, as
\[
\boldsymbol{\eta}(t)=\frac{\cos((t-p)\frac{\pi}{2})}{\|\boldsymbol{\mu}_p\|}\boldsymbol{\mu}_{p}+\frac{\sin((t-p)\frac{\pi}{2})}{\|\boldsymbol{\mu}_{p+1}\|}\boldsymbol{\mu}_{p+1},\qquad \text{ for }p\leq t\leq p+1.
\]
where we regard $\boldsymbol{\mu}_{q+1}$ as $\boldsymbol{\mu}_{1}$. Then $\{\gamma(t)\}_{\{t\in[1,q+1]\}}=\{z_{\mathbf{A}}(\boldsymbol{\eta}(t))\}_{t\in [1,q+1]}$ encloses $z$.

Since $\spanspan(\mathbf{u}_1,\dots,\mathbf{u}_m)\cap S^{m-1}$ is simply connected, we can continuously contract $\{\boldsymbol{\eta}(t)\}_{t\in [1,q+1]}$ to a point in  $\spanspan(\mathbf{u}_1,\dots,\mathbf{u}_m)\cap S^{m-1}$ and the curve under the map $z_\mathbf{A}$ continuously contracts to a point in $z_\mathbf{A}\left(\spanspan(\mathbf{u}_1,\dots,\mathbf{u}_m)\backslash\{0\}\right)$.
However, this is not possible as
$z\notin z_\mathbf{A}\left(\spanspan(\mathbf{u}_1,\dots,\mathbf{u}_m)\backslash\{0\}\right)$ and $\{\gamma(t)\}_{\{t\in[1,q+1]\}}$ has a nonzero winding number around $z$.
We have a contradiction and we conclude  $z\in z_\mathbf{A}\left(\spanspan(\mathbf{u}_1,\dots,\mathbf{u}_m)\backslash\{0\}\right)$.
\end{proof}

\section{SRGs of normal matrices}
We now use Theorem~\ref{mat-main-thm} to fully characterize the SRG of normal matrices.





\begin{figure}[h]
    \centering
\begin{tikzpicture}[scale=0.5]
\begin{scope}
\clip (0,-4.6) rectangle (0,4.6);
\end{scope}

\draw (-2.5,3.5) node {$\cG\left(
\begin{bmatrix}
1&2\\
3&4\\
\end{bmatrix}
\right)=$};

\draw[line width=1pt] (2.5,0.5) circle (2.9155);
\draw[line width=1pt] (2.5,-0.5) circle (2.9155);

\draw [<->] (-2,0) -- (7,0);
\draw [<->] (0,-4) -- (0,4);

\end{tikzpicture}
\begin{tikzpicture}[scale=1]
\begin{scope}
\clip (0,-2.3) rectangle (0,2.3);
\end{scope}

\draw (-1.5,1.75) node {$\cG\left(
\begin{bmatrix}
\frac{1}{2}&2\\
-\frac{1}{2}&\frac{1}{2}\\
\end{bmatrix}
\right)=$};

\draw[line width=1pt] (0.5,1.25) circle (0.75);
\draw[line width=1pt] (0.5,-1.25) circle (0.75);

\draw [<->] (-1,0) -- (2,0);
\draw [<->] (0,-2.25) -- (0,2.25);

\end{tikzpicture}
    \caption{Illustration of Proposition~\ref{mat-2*2}}
    \label{mat_2_circ}
\end{figure}
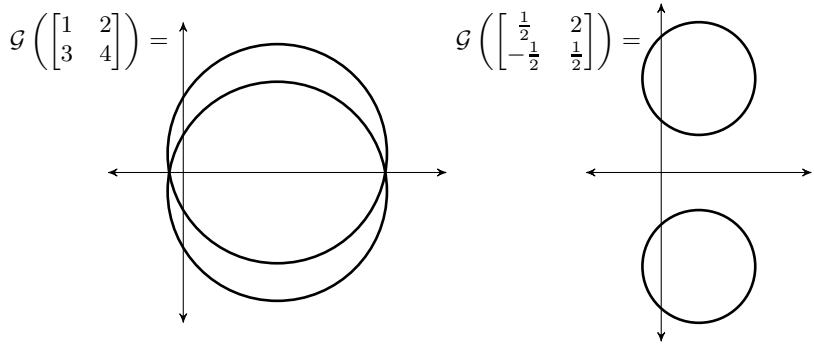

\begin{proposition}
\label{mat-2*2}
Let $A=\left[
\begin{array}{cc}	
a_1&b_1\\
b_2&a_2\\
\end{array}\right]\in \reals^{2\times 2}$.
Then $\mathcal{G}(A)$ consists of two circles centered at $\left(\frac{a_1+a_2}{2},\pm\frac{b_1-b_2}{2}\right)$ with radius $\sqrt{\left(\frac{a_1-a_2}{2}\right)^2+\left(\frac{b_1+b_2}{2}\right)^2}$.
\end{proposition}
\begin{proof}
Let
\[
x_\theta=
\begin{bmatrix}
\cos(\theta)\\\sin(\theta)
\end{bmatrix}\in \reals^2,
\qquad
T\left(\begin{bmatrix}x_1\\x_2\end{bmatrix}\right)=
\begin{bmatrix}
x_1\\|x_2|
\end{bmatrix}.
\]
The stated result follows from 
\[
\cG^+(A)=\{z_A(x_\theta):\theta\in[0,2\pi)\}
\]
and the calculations
\begin{align*}
z_A(x_\theta)
&=
\begin{bmatrix}
\frac{1}{2}\left(a_1+a_2+(a_1-a_2)\cos(2\theta)+(b_1+b_2)\sin(2\theta)\right)\\
\frac{1}{2}
\left|
-b_1+b_2+(b_1+b_2)\cos(2\theta)-(a_1-a_2)\sin(2\theta)
\right|
\end{bmatrix}\\
&=
T\bigg(
\begin{bmatrix}
\frac{a_1+a_2}{2}\\
-\frac{b_1-b_2}{2}
\end{bmatrix}+
\underbrace{
\begin{bmatrix}
\cos(-2\theta)& -\sin (-2\theta)\\
\sin (-2\theta)&\cos(-2\theta)
\end{bmatrix}}_{\text{rotation by $-2\theta$}}
\begin{bmatrix}
\frac{a_1-a_2}{2}\\
\frac{b_1+b_2}{2}
\end{bmatrix}\bigg).
\end{align*}
\end{proof}

\begin{figure}
	\begin{center}
	    \begin{subfigure}[b]{0.42\textwidth}
\begin{tikzpicture}[scale=1.2]
					\begin{scope}
						\clip (-1.25,-2.666666) rectangle (2.45,2.6666666);
						\fill[medgrey] (1,0) circle ({sqrt(5)});
					\end{scope}
					
					\fill[white] (5/8,0) circle (5/8);
					\fill[white] (-1,0) circle (1); 
					\begin{scope}
						\clip (-1.25,-2.3) rectangle (2.5,2.3);
						\fill[white] (27/8,0) circle ({sqrt((19/8)^2+(1/2)^2)}); 
					\end{scope}
					
					\filldraw (-1,1) circle (1.0/1.2pt);
					\filldraw (0,0) circle (1.0/1.2pt);
					\filldraw (1,0.5) circle (1.0/1.2pt);
					\filldraw (2,2) circle (1.0/1.2pt);
					
					\filldraw (-1,-1) circle (1.0/1.2pt);
					\filldraw (0,-0) circle (1.0/1.2pt);
					\filldraw (1,-0.5) circle (1.0/1.2pt);
					\filldraw (2,-2) circle (1.0/1.2pt);
					
					\draw (-1.1,0.85) node {$\lambda_6$};
					\draw (1.2,0.4) node {$\lambda_2$};
					\draw (2.2,2.0) node {$\lambda_4$};
					\draw (-1.1,-0.85) node {$\lambda_7$};
					\draw (1.2,-0.4) node {$\lambda_3$};
					\draw (2.2,-2.0) node {$\lambda_5$};
					\draw (-0.12,0.15) node {$\lambda_1$};
					
					\draw[<->] (-1.2,0) -- (2.4,0);					
				\end{tikzpicture}
        \caption{SRG of an $n\times n$ normal matrix with one distinct real eigenvalue and three distinct complex conjugate eigenvalue pairs.}
	\label{fig:normal-SRG}
    \end{subfigure}
    \hspace{0.13in}
    	    \begin{subfigure}[b]{0.54\textwidth}
		\begin{tikzpicture}[scale=1]
				\clip (-3.565,-3.2) rectangle (3.35,3.2);
				\draw (-3.2,0.15) node[] {$\lambda_1$};
				\fill[medgrey] (0,0) circle (3);
				
				\fill[white] (-2.5,0) circle (0.5);
				\draw (-2.2,0.15) node[] {$\lambda_2$};

				\fill[white] (-1.3,0) circle (0.7);
				\draw (-0.8,0.15) node[] {$\lambda_3$};

				\fill[white] (0,0) circle (0.6);
				\draw (0.4,0.15) node[] {$\lambda_4$};
				
				\fill[white] (0.9,0) circle (0.3);
				\draw (0.98,0.15) node[] {$\lambda_5$};
				
				\fill[white] (2.1,0) circle (0.9);
				\draw (2.8,0.15) node[] {$\lambda_6$};
				
				\filldraw (-3,0) circle ({1.0/1pt});
				\filldraw (-2,0) circle ({1.0/1pt});
				\filldraw (-0.6,0) circle ({1.0/1pt});
				\filldraw (0.6,0) circle ({1.0/1pt});
				\filldraw (1.2,0) circle ({1.0/1pt});
				\filldraw (3,0) circle ({1.0/1pt});

				\draw [<->] (-3.6,0) -- (3.3,0);
				\end{tikzpicture}
	        \caption{SRG of an $n\times n$ symmetric matrix with distinct eigenvalues $\lambda_1<\lambda_2<\dots<\lambda_6$. \phantom{sadfiojasd} \phantom{oifjasdoifj} \phantom{aoisdfj} \phantom{oiasdj} \phantom{sadfiojasd} \phantom{oifjasdoifj} \phantom{aoisdfj} \phantom{oiasdj}}
	\label{fig:symm-srg}
    \end{subfigure}%
	\end{center}
	\caption{
	Illustration of Theorem~\ref{mat-main-thm2}.
	For normal matrices, multiplicity of eigenvalues do not affect the SRG.}
\end{figure}
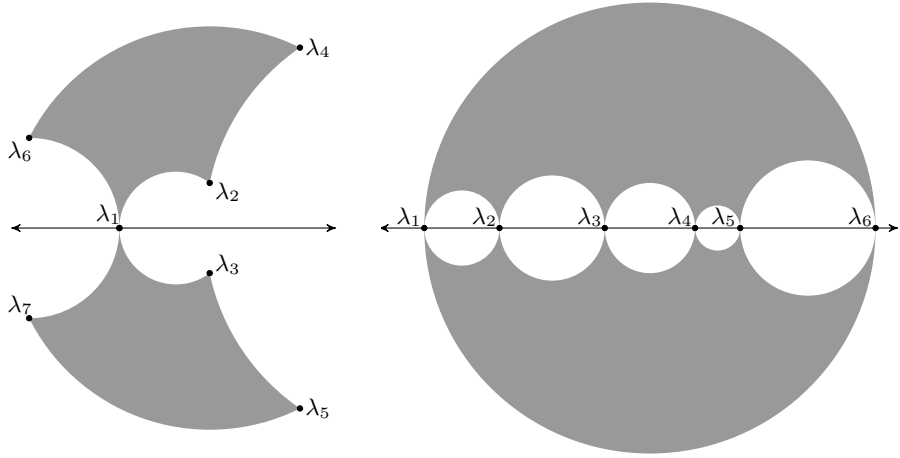

\begin{proposition}
\label{mat_orth_inv}
A matrix's SRG is invariant under orthogonal similarity transforms.
\end{proposition}
\begin{proof}
Let $A\in \reals^{n\times n}$. Let $Q\in \reals^{n\times n}$ be orthogonal. 
For any nonzero $x\in \reals^n$, we have
\begin{align*}
z_{QAQ^T}(x)
&=\frac{\|QAQ^Tx\|}{\| x\|}\exp[i\angle (QAQ^Tx,x)]\\
&=\frac{\|AQ^Tx\|}{\| x\|}\exp[i\angle (AQ^Tx,Q^Tx)]
=
z_A(Q^Tx).
\end{align*}
Therefore, 
\begin{align*}
\cG(QAQ^T)
&=
\left\{z_{QAQ^T}(x),\overline{z_{QAQ^T}}(x):
x\in \reals^n,\,x\ne 0
\right\}\\&
=
\left\{z_{A}(Qx),\overline{z_{A}}(Qx):
x\in \reals^n,\,x\ne 0
\right\}\\&
=
\left\{z_{A}(x),\overline{z_{A}}(x):
x\in \reals^n,\,x\ne 0
\right\}=\cG(A).
\end{align*}
\end{proof}


\begin{theorem}
\label{mat-main-thm2}
If $A$ is normal, 
then $\mathcal{G}^+(A)=\D(\Lambda(A)\cap\complex^+)$.
\end{theorem}

\begin{proof}
A normal matrix is orthogonally similar to the real block-diagonal matrix
\[
\begin{bmatrix}
a_1&b_1&&&&&&\\
-b_1&a_1&&&&&&\\
&&\ddots&&&&&\\
&&&a_k&b_k&&&\\
&&&-b_k&a_k&&&\\
&&&&&\lambda_{k+1}&&\\
&&&&&&\ddots&\\
&&&&&&&\lambda_{m}\\
\end{bmatrix}.
\]
Propositions~\ref{mat-2*2} tells us 
\[
\cG^+\left(\begin{bmatrix}
a_j&b_j\\
-b_j&a_j\\
\end{bmatrix}\right)=
\{a_j+|b_j|i\}=\Lambda\left(\begin{bmatrix}
a_j&b_j\\
-b_j&a_j\\
\end{bmatrix}\right)\cap\complex^+
\]
for $j=1,\dots,k$.
We conclude the stated result with Theorem~\ref{mat-main-thm2} and Proposition~\ref{mat_orth_inv}.
\end{proof}

\begin{corollary}
\label{cor:symm-srg}
 Let $A\in \reals^{n\times n}$ be symmetric, and let $\lambda_1<\lambda_2<\dots<\lambda_m$ be the distinct (real) eigenvalues of $A$.
If $m=1$, then $\mathcal{G}^+(A)=\{\lambda_1\}$.
If $m\ge 2$, then
\[
\mathcal{G}(A)=\disk(\lambda_1,\lambda_m)\backslash\bigcup\limits_{i=1}^{ m-1}\disk^\circ(\lambda_i,\lambda_{i+1}).
\]
\end{corollary}


%
%
%

\bibliographystyle{siamplain}
\bibliography{my_srg}

\end{document}